\documentclass[12pt,reqno]{amsart}

\usepackage{hyperref}
\hypersetup{colorlinks=true, linkcolor=red, citecolor=blue}

\usepackage{amssymb,amsmath,amsfonts,latexsym,mathrsfs}
\usepackage{bm}
\usepackage{mathtools}
\usepackage{enumerate}
\usepackage{array,graphics,color}
\allowdisplaybreaks
\usepackage{mathrsfs}
\usepackage{mathtools}
\usepackage{blkarray, bigstrut}
\usepackage{csquotes}
\usepackage{ulem}

\numberwithin{equation}{section}

\newtheorem{dfn}{Definition}[section]
\newtheorem{thm}[dfn]{Theorem}
\newtheorem{lma}[dfn]{Lemma}
\newtheorem{ppsn}[dfn]{Proposition}
\newtheorem{crl}[dfn]{Corollary}

\newtheorem{rmrk}[dfn]{Remark}

\newcommand{\cir}{\mathbb{T}}
\newcommand{\R}{\mathbb{R}}
\newcommand{\C}{\mathbb{C}}
\newcommand{\N}{\mathbb{N}}	
\newcommand{\Z}{\mathbb{Z}}
\newcommand{\hilh}{\mathcal{H}}

\newcommand{\bh}{\mathcal{B}(\hilh)}

\newcommand{\Sp}{\mathcal{S}}

\newcommand{\Tr}{\operatorname{Tr}}

\begin{document}
	
\title{Higher order $\Sp^{p}$-differentiability: The unitary case}

\author[A. Chattopadhyay] {Arup Chattopadhyay}
\address{Department of Mathematics, Indian Institute of Technology Guwahati, Guwahati, 781039, India}
\email{arupchatt@iitg.ac.in, 2003arupchattopadhyay@gmail.com}

\author[C. Coine]{Cl\'ement Coine}
\address{Normandie Univ, UNICAEN, CNRS, LMNO, 14000 Caen, France}	
\email{clement.coine@unicaen.fr}

\author[S. Giri]{Saikat Giri}
\address{Department of Mathematics, Indian Institute of Technology Guwahati, Guwahati, 781039, India}	\email{saikat.giri@iitg.ac.in, saikatgiri90@gmail.com}

\author[C. Pradhan]{Chandan Pradhan}
\address{Department of Mathematics, Indian Institute of Science, Bangalore, 560012, India}
\email{chandan.pradhan2108@gmail.com}	
	
	

\begin{abstract}
Consider the set of unitary operators on a complex separable Hilbert space $\hilh$, denoted as $\mathcal{U}(\hilh)$. Consider $1<p<\infty$. We establish that a function $f$ defined on the unit circle $\cir$ is $n$ times continuously Fr\'echet $\Sp^p$-differentiable at every point in $\mathcal{U}(\hilh)$ if and only if  $f\in C^n(\cir)$. Take a function $U :\R\rightarrow\mathcal{U}(\hilh)$ such that the function $t\in\R\mapsto U(t)-U(0)$ takes values in $\Sp^{p}$ and is $n$ times continuously $\Sp^{p}$-differentiable on $\R$. Consequently, for $f\in C^n(\cir)$, we prove that $f$ is $n$ times continuously G\^ateaux $\mathcal{S}^p$-differentiable at $U(t)$. We provide explicit expressions for both types of derivatives of $f$ in terms of multiple operator integrals. In the domain of unitary operators, these results closely follow the $n$th order successes for self-adjoint operators achieved by the second author, Le Merdy, Skripka, and Sukochev. Furthermore, as for application, we derive a formula and $\Sp^{p}$-estimates for operator Taylor remainders for a broader class of functions.  Our results extend those of Peller, Potapov, Skripka, Sukochev and Tomskova.
\end{abstract}
\maketitle

	
\section{Introduction}\label{Sec1}
Beginning the journey into function theory, especially when delving into operator functions, the concept of differentiability takes on crucial significance. Early breakthroughs in comprehending the differentiability of operator functions were achieved in \cite{DaKr56}, subject to rigorous conditions imposed on both functions and operators. These initial insights underwent significant refinement and expansion in subsequent works, such as \cite{AzCaDoSu09, ArBaFr90, BiSo3, DePSu04, KiSh04, KiPoShSu14, Pe00, Pe06, PoSkSu16, St77}, responding to advancements in perturbation theory. Despite these advancements, certain challenges persisted. This paper addresses one of the open problems.

Throughout the article we consider $\hilh$ to be a complex separable Hilbert space and let $\bh$ and $\mathcal{U}(\hilh)$ denote respectively the space of bounded linear operators, and the space of unitary operators on $\hilh$. For $p\in[1,\infty)$, the $p$th Schatten-von Neumann class is 
\begin{align*}
\Sp^{p}(\hilh)=\left\{A\in\bh:\|A\|_{p}:=\Tr(|A|^{p})^{1/p}<\infty\right\},
\end{align*}
$\Sp^{p}_{sa}(\hilh)$ is the subset of self-adjoint elements of $\Sp^{p}(\hilh)$, and $\Tr$ is the canonical trace on the trace class ideal  $\Sp^{1}(\hilh)$. The unsuffixed norm $\|\cdot\|$ will always mean the operator norm.

For self-adjoint operators, Le Merdy and Skripka \cite{LeSk20} established the $n$th order Fr\'echet and G\^ateaux $\Sp^{p}$-differentiability ($1<p<\infty$) of the operator function $t\in\R\mapsto f(A+tB)-f(A)$ for an $n$ times continuously differentiable function $f$. Subsequently, in \cite{Co22}, the second author extended this class from $n$ times continuously differentiable to $n$ times differentiable functions concerning $n$th order G\^ateaux $\Sp^{p}$-differentiability. The analogous problem of studying the differentiability of an operator function for unitary operators has also been investigated by various experts in this field.

Let $U$ be a unitary operator, and $A$ be a bounded self-adjoint operator acting on $\hilh$. Define $$U(s)=e^{isA}U,~ s\in\R.$$ 
In \cite{PoSkSu16}, for any function $f$ belonging to the Besov class $B^n_{\infty 1}(\cir)$ on the unit circle $\cir$, where $n(\ge 2)$ is a natural number, Potapov, Skripka, and Sukochev gave the existence of the $n$th order G\^ateaux derivative of $f$ at $U$, $\frac{d^n}{ds^n}\big|_{s=t}f(U(s))$, in the operator norm. Moreover, they expressed the value as a linear combination of multiple operator integrals with respect to divided differences $f^{[1]}, f^{[2]},\ldots,f^{[n]}$ of the function $f$ (see \cite[Theorem 3.1]{PoSkSu16}), correcting the formula earlier obtained in \cite{Pe06}. The case $n=1$ was addressed in \cite{BiSo3,Pe85}. A recent proof in \cite[Theorem 3.6]{PoSkSuTo17} established that for a function $f\in C^{n}(\cir)$ and a finite-dimensional Hilbert space $\hilh=\ell^{2}_{d}$, the function is $n-1$ times G\^ateaux $\Sp^{p}$-differentiable at $U$.

\smallskip

\noindent{\bf New results, novelty, and methodology:}  In this article, we answer several questions on higher order differentiability of operator functions in $\|\cdot\|_{p}$, $1<p<\infty$. The exponent $p$ in this paper is reserved for a number in the interval $p\in(1,\infty)$.

We prove (see Theorem \ref{Frechetdiff}) that $f$ is $n$ times continuously Fr\'echet $\Sp^{p}$-differentiable at every unitary operator $U\in\mathcal{U}(\hilh)$ if and only if $f\in C^{n}(\cir)$.  Consequently, we prove (see Theorem \ref{MainTheorem1}) $f$ is $n$ times continuously G\^ateaux $\Sp^{p}$-differentiable under the assumption $f\in C^{n}(\cir)$ by extending the relative results of \cite{Pe06, PoSkSu16}, which specifically addressed the G\^ateaux differentiability of $f\in B^{n}_{\infty 1}(\cir)$. Let $U: \mathbb{R} \rightarrow \mathcal{U}(\hilh)$ be a function such that 
$$\tilde{U}:t\in\R\mapsto U(t)-U(0)\in\Sp^{p}(\hilh)$$
is an $n$ times $\mathcal{S}^p$-differentiable function on $\mathbb{R}$. Then, for $f\in C^n(\cir)$, we establish the $n$ times $\Sp^p$-differentiability of
$$\varphi:t\in\R\mapsto f(U(t))-f(U(0))\in\Sp^{p}(\hilh).$$ In addition, if $\tilde{U}^{(n)}$ is continuous then $\varphi^{(n)}$ is also continuous in $\Sp^{p}(\hilh)$ (see Theorem \ref{MainTheorem1}(i)). It is crucially noted in the literature (see, e.g., \cite{PoSkSu16, Sk17,PoSkSuTo17}) that the higher order G\^ateaux derivatives of an operator function play a fundamental role in examining the behavior of the higher order Taylor remainder of the operator function. 

\smallskip

\noindent Consider the $n$th order Taylor remainder
\begin{align}\label{taylorremunigen}
R_{n,f,U}(t):=f(U(t))-f(U(0))-\sum_{m=1}^{n-1} \dfrac{1}{m!}\varphi^{(m)}(0).
\end{align}
In particular, the following estimate of the above Taylor remainder \eqref{taylorremunigen} is already obtained in \cite{PoSkSuTo17} for $1<n<p<\infty$, $U(t)=e^{itA}U$ with $A\in \Sp^p_{sa}(\hilh)$,
\begin{align}\label{introremainderest}
\|R_{n,f,U}(1)\|_{\frac{p}{n}}\le \tilde{c}_{p,n}\sum_{m=1}^{n}\|f^{(m)}\|_{\infty}\|A\|_p^n
\end{align}
for every $f\in C^{n}(\cir)$, where $\tilde{c}_{p,n}$ is a positive constant depending on $p$ and $n$. 

It should be noted that for $f\in\mathcal{G}_{n}(\cir)\subsetneq C^n(\cir)$, an explicit formula for \eqref{taylorremunigen} with $U(t)=e^{itA}U$ is known in terms of the integration of certain multiple operator integrals (see \cite[Theorem 4.1]{PoSkSu16}). In \cite{Topre} (detailed proof can be found in \cite[Theorem 3.7]{PoSkSuTo17}), the author established a simple formula for the aforementioned Taylor remainder that corresponds to $f\in C^n(\cir)$ as a linear combination of multiple operator integrals (without involving the integration of multiple operator integrals), with respect to divided differences $f^{[1]}, f^{[2]},\ldots,f^{[n]}$, constrained to a finite-dimensional Hilbert space $\hilh=\ell^{2}_{d}$. Here, for a complex separable Hilbert space (not necessarily finite-dimensional), we establish the same formula for \eqref{taylorremunigen} for any $f\in C^n(\cir)$ (see, e.g., Theorem \ref{MainTheorem1}(ii), Corollary \ref{MainCorollary1}). Consequently, our results lead to a direct proof of \eqref{introremainderest} (see Corollary \ref{MainCorollary1}).

In conclusion, let us discuss the methodology employed in this paper. To achieve the results mentioned earlier, we will delve into establishing crucial properties of multiple operator integrals throughout the article. Our approach leans towards the construction of operator integrals outlined in \cite{CoLeSu21}, differing from those presented in \cite{Pe06, PoSkSu16, Sk17}. Specifically, for $1 < p < \infty$, $n\in\N$, and unitary operators $U_1,\ldots,U_n$ acting on the Hilbert space $\hilh$, we heavily use the $\Sp^{p}$-boundedness of operator integrals
\begin{align*}
\left\|\Gamma^{U_1,\ldots,U_{n+1}}(f^{[n]})\right\|_{\mathcal{B}_{n}(\Sp^p)}\leq c_{p,n} \left\|f^{(n)}\right\|_{\infty},
\end{align*}
for $f\in C^{n}(\cir)$, extending the bound of operator integrals considered in \cite[Theorem 3.3]{Sk17}. The importance of this construction compared to other constructions of multiple operator integrals \cite{AzCaDoSu09,BiSo1,DaKr56,Pe06,PoSkSu13,PoSkSu16} are discussed in the next Section \ref{Sec2}.

The article is organized as follows: Apart from the Introduction, it consists of two sections. In Section \ref{Sec2}, we will mainly concern ourselves with the properties of operator integrals. Firstly, in Subsection \ref{Subsec1}, we recall as well as establish the important properties of multiple operator integrals, and in Subsection \ref{Subsec2}, we establish the higher order perturbation formula for differences of multiple operator integrals. Section \ref{Sec3} addresses the differentiability of functions of unitary operators in $\Sp^{p}(\hilh)$ for $1<p<\infty$ by listing the statements of the main results in Subsection \ref{Subsec3}, then some auxiliary lemmas in Subsection \ref{Subsec4}, and finally providing the proof of the main results in Subsection \ref{Subsec5}.

\section{Multiple operator integration}\label{Sec2}

In this section, we recall the definition of  multiple operator integrals introduced in \cite{Pav69} and developed in \cite{CoLeSu21}, as well as derive its important properties that underline our main results.

\subsection{Multiple operator integrals associated to unitary operators}\label{Subsec1}\hfill\\

Progress in investigating the operator smoothness and differentiability relies on methodologies known as \enquote{multiple operator integration}, evolving since \cite{DaKr56}. There exist two primary approaches to multiple operator integration on Schatten classes, as detailed in \cite{SkToBook}. For our purposes, we adopt the definition of multiple operator integration developed in \cite{CoLeSu21}. While alternative constructions of multiple operator integrals exist \cite{AlNaPe2016, AlPe2016, AzCaDoSu09,BiSo1,DaKr56,Pe06,PoSkSu13,PoSkSu16}, they are tailored to smaller sets of symbols and thus do not align with the breadth of this paper's scope.

We fix the following notations and conventions to be used throughout the paper. Let $\mathcal{B}_{n}(X_1\times\cdots\times X_n,Y)$ denote the space of bounded $n$-linear operators mapping the Cartesian product $X_1\times\cdots\times X_n$ of Banach spaces $X_1,\ldots,X_n$ to a Banach space $Y$, that is, the space of $n$-linear mappings $T:X_1\times\cdots\times X_n\to Y$ such that
\begin{align*}
\|T\|_{\mathcal{B}_{n}(X_1\times\cdots\times X_n,Y)}:=\sup_{\|e_i\|\le 1,\,1\le i\le n}\|T(e_1,\ldots,e_n)\|<\infty.
\end{align*}
In the case when $X_1=\cdots=X_n=Y$, we will simply denote $\mathcal{B}_{n}(X_1\times\cdots\times X_n,Y)$ by $\mathcal{B}_{n}(Y)$. If $T\in\bh$, then for any $k\in\N$, we use the notation $(T)^{k}=\underbrace{T,\ldots,T}_{\text{$k$}}$. We denote the unit circle of $\C$ by $\cir$ and the class of continuous functions on $\cir$ by $C(\cir)$. For $f\in C(\cir)$, by its derivative at $z_{0}\in\cir$, we understand the limit
\begin{align}\label{ref14}
f'(z_0):=\lim_{\cir\ni z\to z_{0}}\dfrac{f(z)-f(z_{0})}{z-z_{0}},
\end{align}
provided it exists. For any $n\in\N$, denote by $C^n(\cir)$ the space of $n$ times continuously differentiable functions on $\cir$ in the sense of \eqref{ref14}. $\text{Lip}(\cir)$ $(\subset C(\cir))$ denotes the space of all Lipschitz functions $f:\cir\to\C$. Finally we let $D^n(\R,\Sp^{p}(\hilh))$ (respectively $C^n(\R,\Sp^{p}(\hilh))$) be the space of $n$ times differentiable (respectively continuously differentiable) functions $\varphi:\R\to\Sp^p(\hilh)$ with derivatives denoted by $\varphi^{(j)}:\R\to\Sp^p(\hilh),~j=1,\ldots,n.$

Let $U$ be a unitary operator on $\hilh$. Denote its spectrum by $\sigma(U)$ and its spectral measure by $E^U$, defined on Borel subsets of $\sigma(U)$. Let $\lambda_{U}$ be a scalar-valued spectral measure for $U$, which is a positive finite measure on the Borel subsets of  $\sigma(U)$ such that $E^U$ and $\lambda_{U}$ have the same sets of measure zero. Such a measure exists, due to the separability assumption on $\hilh$. The reader is directed to Section 15 in \cite{ConwayBook} and Section 2.1 in \cite{CoLeSu21} for more details. The Borel functional calculus for $U$ takes any bounded Borel function $f:\sigma(U)\to\C$ to the bounded operator
\begin{align*}
f(U):=\int_{\sigma(U)}^{}f(z)~dE^{U}(z).
\end{align*}
This operator only depends on the class of $f$ in $L^{\infty}(\lambda_{U})$.
Moreover, according to \cite[Theorem 15.10]{ConwayBook}, it induces a $w^{\ast}$-continuous $\ast$-representation
\begin{align*}
f\in L^{\infty}(\lambda_{U})\mapsto f(U)\in\bh.
\end{align*}
Let $n\in\N,~n\ge 2$, and let $U_1,\ldots,U_n$ be unitary operators on $\hilh$ with scalar-valued spectral measures $\lambda_{U_1},\ldots,\lambda_{U_n}$.

\begin{dfn}\label{DefMOI}
Let
\begin{align*}
\Gamma:L^{\infty}(\lambda_{U_1})\otimes\cdots\otimes L^{\infty}(\lambda_{U_n})\to\mathcal{B}_{n-1}(\Sp^2(\hilh))
\end{align*}
be the unique linear map such that for any $f_{i}\in L^{\infty}(\lambda_{U_i}),~1\le i\le n$, and for any $K_1,\ldots,K_{n-1}\in\Sp^2(\hilh)$,
\begin{align}
&\left[\Gamma(f_1\otimes\cdots\otimes f_n)\right](K_1,\ldots,K_{{n-1}})=f_1(U_1)K_1f_2(U_2)K_2\cdots f_{n-1}(U_{n-1})K_{n-1}f_n(U_n).
\end{align}
According to {\normalfont\cite[Proposition 3.4]{CoLeSu21}}, $\Gamma$ extends to a unique $w^{\ast}$-continuous and contractive map
\begin{align*}
\Gamma^{U_1,\ldots,U_n}:L^{\infty}\left(\prod_{i=1}^{n}\lambda_{U_i}\right)\to\mathcal{B}_{n-1}(\Sp^2(\hilh)).
\end{align*}
For $\varphi\in L^{\infty}\left(\prod_{i=1}^{n}\lambda_{U_i}\right)$, the transformation $\Gamma^{U_1,\ldots,U_n}(\varphi)$ is called multiple operator integral associated to $U_1
,\ldots,U_n$ and $\varphi$.\\
Let $\varphi:\cir^{n}\to\C$ be a bounded Borel function, we let $\widetilde{\varphi}$ be the class of its restriction $\varphi\vert_{\sigma(U_1)\times\cdots\times\sigma(U_n)}$ in $L^{\infty}\left(\prod_{i=1}^{n}\lambda_{U_i}\right)$. Then the $(n-1)$-linear map $\Gamma^{U_1,\ldots,U_n}(\widetilde{\varphi})$ will be simply denoted by 
\begin{align*}
\Gamma^{U_1,\ldots,U_n}(\varphi):\Sp^{2}(\hilh)\times\cdots\times\Sp^{2}(\hilh)\to\Sp^{2}(\hilh)
\end{align*}
in the sequel.
\end{dfn}

The $w^{\ast}$-continuity of $\Gamma^{U_1,\ldots,U_n}$ means that if a net $(\varphi_{i})_{i\in I}$ in $L^{\infty}\left(\prod_{i=1}^{n}\lambda_{U_i}\right)$ converges to $\varphi\in L^{\infty}\left(\prod_{i=1}^{n}\lambda_{U_i}\right)$ in the $w^{\ast}$-topology, then for any $K_1
,\ldots,K_{n-1}\in\Sp^{2}(\hilh)$, the net
\begin{align*}
\left(\left[\Gamma^{U_1,\ldots,U_n}(\varphi_i)\right](K_1,\ldots,K_{n-1})\right)_{i\in I}
\end{align*}
converges to $\left[\Gamma^{U_1,\ldots,U_n}(\varphi)\right](K_1,\ldots,K_{n-1})$ weakly in $\Sp^{2}(\hilh)$.

Let $p_1,\ldots,p_{n-1},p\in[1,\infty)$ and $\varphi\in L^{\infty}\left(\prod_{i=1}^{n}\lambda_{U_i}\right)$. We will write
\begin{align}
\Gamma^{U_1,\ldots,U_n}(\varphi)\in\mathcal{B}_{n-1}\left(\Sp^{p_1}(\hilh)\times\cdots\times\Sp^{p_{n-1}}(\hilh),\Sp^{p}(\hilh)\right)
\end{align}
if the multiple operator integral $\Gamma^{U_1,\ldots,U_n}(\varphi)$ defines a bounded $(n-1)$-linear mapping 
\begin{align*}
\Gamma^{U_1,\ldots,U_n}(\varphi):\left(\Sp^2(\hilh)\cap\Sp^{p_1}(\hilh)\right)\times\cdots\times\left(\Sp^2(\hilh)\cap\Sp^{p_{n-1}}(\hilh)\right)\to\Sp^{p}(\hilh),
\end{align*}
where $\Sp^2(\hilh)\cap\Sp^{p_i}(\hilh)$ is equipped with the $\|\cdot\|_{p_i}$-norm. By the density of $\Sp^2(\hilh)\cap\Sp^{p_i}(\hilh)$ into $\Sp^{p_i}(\hilh)$, this mapping has an (necessarily unique) extension 
\begin{align}
\label{MOI}\Gamma^{U_1,\ldots,U_n}(\varphi):\Sp^{p_1}(\hilh)\times\cdots\times\Sp^{p_{n-1}}(\hilh)\to\Sp^{p}(\hilh).
\end{align}
In the case when $p_1=\cdots=p_{n-1}=p$, we will simply write $\Gamma^{U_1,\ldots,U_n}(\varphi)\in\mathcal{B}_{n-1}(\Sp^{p}(\hilh))$.

The crucial point in the construction leading to Definition \ref{DefMOI} is the $w^{\ast}$-continuity of $\Gamma^{U_1,\ldots,U_n}$, which allows to reduce various computations to elementary tensor product manipulations, for which certain equations are straightforward to establish. See \cite{CoLeSkSu17} and \cite{Co22} for illustrations.

The following result is instrumental in demonstrating the $\Sp^p$-boundedness of certain multiple operator integrals. It can be proven by following similar lines of argument as presented in the proof of \cite[Lemma 2.3]{Co22}.

\begin{lma}\label{MOIProperty1}
Let $p_{1},\ldots,p_{n-1},p\in(1,\infty)$, and $n\ge 2$ be an integer. Let $U_{1},\ldots,U_{n}$ be unitary operators on $\hilh$ and $(\varphi_{k})_{k\ge 1}, \varphi\in L^{\infty}(\lambda_{U_1}\times\cdots\times\lambda_{U_n})$. Assume that $(\varphi_{k})_{k}$ is $w^{\ast}$-convergent to $\varphi$ and that $(\Gamma^{U_1,\ldots,U_n}(\varphi_{k}))_{k\ge 1}\subset\mathcal{B}_{n-1}(\Sp^{p_{1}}(\hilh)\times\cdots\times\Sp^{p_{n-1}}(\hilh),\,\Sp^{p}(\hilh))$ is bounded. Then 
\begin{align*}
\Gamma^{U_1,\ldots,U_n}(\varphi)\in\mathcal{B}_{n-1}(\Sp^{p_{1}}(\hilh)\times\cdots\times\Sp^{p_{n-1}}(\hilh),\,\Sp^{p}(\hilh))
\end{align*} 
with 
\begin{align*}
\left\|\Gamma^{U_1, \ldots, U_n}(\varphi)\right\|_{\mathcal{B}_{n-1}(\Sp^{p_{1}}\times\cdots\times\Sp^{p_{n-1}},\, \Sp^{p})}\le\liminf_{k}\left\|\Gamma^{U_1, \ldots, U_n}(\varphi_{k})\right\|_{\mathcal{B}_{n-1}(\Sp^{p_{1}}\times\cdots\times\Sp^{p_{n-1}},\,\Sp^{p})}
\end{align*}
and for any $K_{i}\in\Sp^{p_i}(\hilh),~1\le i\le n-1,$
\begin{align*}
\left[\Gamma^{U_1, \ldots, U_n}(\varphi_{k})\right](K_1,\ldots,K_{n-1}) \xrightarrow[k\to\infty]{}\left[\Gamma^{U_1, \ldots, U_n}(\varphi)\right](K_1,\ldots,K_{n-1})
\end{align*}
weakly in $\Sp^{p}(\hilh)$.
\end{lma}

\subsection{Higher order perturbation formula}\label{Subsec2}\hfill\\

This section is devoted to obtaining an important result on boundedness of multiple operator integrals associated to divided differences $f^{[n]}$ in the case when $f\in C^{n}(\cir)$, which will justify that all the operators appearing in the sequel are well-defined. We also prove a higher order perturbation formula for differences of multiple operator integrals.

We first recall the definition of the divided difference. The zeroth order divided difference of $f$ is the function itself, that is, $f^{[0]}:=f$. Let $f\in C^{1}(\cir)$. The divided difference of the first order $f^{[1]}:\cir^{2}\to\C$ is defined by
\begin{align*}
f^{[1]}(\lambda_1,\lambda_2):=\begin{cases*}
\frac{f(\lambda_1)-f(\lambda_2)}{\lambda_1-\lambda_2}\ \ \ \text{if}\ \lambda_1\neq\lambda_2\\
f'(\lambda_1)\ \ \ \ \ \ \ \ \text{if}\ \lambda_1=\lambda_2,
\end{cases*}\quad\lambda_1,\lambda_2\in\cir.
\end{align*}
The function $f^{[1]}$ belongs to $C(\cir^{2})$. If $n\ge 2$ and $f\in C^{n}(\cir)$, the divided difference of $n$th order $f^{[n]}:\cir^{n+1}\to\C$ is defined recursively by
\begin{align*}
f^{[n]}(\lambda_1,\lambda_2,\ldots,\lambda_{n+1}):=\begin{cases*}
\frac{f^{[n-1]}(\lambda_1,\lambda_{3},\ldots,\lambda_{n+1})-f^{[n-1]}(\lambda_2,\lambda_{3},\ldots,\lambda_{n+1})}{\lambda_1-\lambda_2}\ \ \ \text{if}\ \lambda_1\neq\lambda_2\\
\partial_{1}f(\lambda_1,\lambda_{3},\ldots,\lambda_{n+1})\ \ \ \ \ \ \ \ \ \ \ \ \ \ \ \ \ \ \ \ \,\text{if}\ \lambda_1=\lambda_2,
\end{cases*}
\end{align*}
for all $\lambda_1,\ldots,\lambda_{n+1}\in\cir$, where $\partial_{i}$ stands for the partial derivative with respect to the $i$th variable. Moreover the function $f^{[n]}\in C(\cir^{n+1})$.

We are now ready to establish the bound of the multiple operator integral \eqref{MOI} associated with the symbol $\varphi=f^{[n]}$ for $1<p,p_{j}<\infty,~j=1,\ldots,n$ and $f\in C^{n}(\cir)$. This class extends earlier known function classes used to give bounds to the multiple operator integral considered in \cite{Sk17}. We also refer to \cite{AlNaPe2016} and \cite{AlPe2016} for a different kind of $\mathcal{S}^p$-estimates in the case when the symbol is in the Haagerup tensor product of $L^{\infty}$-spaces. The result of Theorem \ref{MOISpBound} below is outlined in \cite[Remark 4.3.20]{SkToBook} without proof. For our purposes, we will state the theorem here and give a simple proof.

\begin{thm}\label{MOISpBound}
Let $n\in\N$ and $f\in C^{n}(\cir)$. Let $1<p,p_{j}<\infty$, $j=1,\ldots,n$ be such that $\frac{1}{p}=\frac{1}{p_{1}}+\cdots+\frac{1}{p_{n}}$. Let $U_1, \ldots, U_{n+1}$ be unitary operators on $\mathcal{H}$. Then
$$\Gamma^{U_1,\ldots,U_{n+1}}(f^{[n]})\in\mathcal{B}_{n}(\Sp^{p_{1}}(\hilh)\times\cdots\times\Sp^{p_{n}}(\hilh),\Sp^p(\hilh))$$
and there exists $c_{p,n}>0$ depending only on $p$ and $n$ such that for any $K_{i}\in\Sp^{p_{i}}(\hilh),~1\le i\le n$,
\begin{align}\label{MOIBound1}
\left\|\left[\Gamma^{U_1,\ldots,U_{n+1}}(f^{[n]})\right](K_{1},\ldots,K_{n})\right\|_{p}\leq c_{p,n}\left\|f^{(n)}\right\|_{\infty}\|K_{1}\|_{p_{1}}\cdots\|K_{n}\|_{p_{n}}.
\end{align}
In particular, $\Gamma^{U_1,\ldots,U_{n+1}}(f^{[n]})\in\mathcal{B}_{n}(\Sp^{p}(\hilh))$, with
\begin{align}\label{MOIBound2}
\left\|\Gamma^{U_1,\ldots,U_{n+1}}(f^{[n]})\right\|_{\mathcal{B}_{n}(\Sp^{p})}\leq c_{p,n} \left\|f^{(n)}\right\|_{\infty}.
\end{align}
\end{thm}

\begin{proof} Let $f\in C^{n}(\cir)$. Define, for any $k\ge 1$, $\varphi_k:=f\ast F_k$, where $F_k$ is the Fej\'er kernel. By the definition of the Fej\'er kernel, the function $\varphi_{k}$ a trigonometric polynomial. Then we have that 
$$\left\|\varphi_k^{(m)}-f^{(m)}\right\|_{\infty} \to 0,\quad k\to\infty,$$
for $0\le m\le n$, which, according to \cite[Lemma 3.2]{Sk17}, further implies 
\begin{align}\label{ref1}
\left\|\varphi_k^{[m]}-f^{[m]}\right\|_{L^{\infty}(\cir^{m+1})}\to 0,\quad k\to\infty
\end{align}  
for all $0\le m\le n$. By \cite[Theorem 3.6]{Sk17}, there exists a constant $c_{p,n}>0$ such that, for every $k\ge 1$,
\begin{align}\label{MOIBound3}
\left\|\Gamma^{U_1,\ldots,U_{n+1}}(\varphi_k^{[n]}) \right\|_{\mathcal{B}_{n}(\Sp^{p_{1}}\times\cdots\times\Sp^{p_{n}},\Sp^p)}\le c_{p,n} \left\|\varphi_k^{(n)}\right\|_{\infty}.
\end{align}
Since $\varphi_{k}^{(n)}=F_{k}\ast f^{(n)}$ and $\|F_{k}\|_{1}\le 1$ for all $k\in\N$, from Young's inequality it follows that
\begin{align}\label{ref2}
\|\varphi_{k}^{(n)}\|_{\infty}\le\|F_{k}\|_{1}\|f^{(n)}\|_{\infty}\le\|f^{(n)}\|_{\infty}.
\end{align}
Therefore, by Lemma \ref{MOIProperty1}, \eqref{ref1}, \eqref{MOIBound3} and \eqref{ref2} we deduce that
$$\Gamma^{U_1,\ldots,U_{n+1}}(f^{[n]})\in\mathcal{B}_{n}(\Sp^{p_{1}}(\hilh)\times\cdots\times\Sp^{p_{n}}(\hilh),\Sp^p(\hilh))$$
with 
\begin{align*}
\left\|\Gamma^{U_1,\ldots,U_{n+1}}(f^{[n]})\right\|_{\mathcal{B}_{n}(\Sp^{p_{1}}\times\cdots\times\Sp^{p_{n}},\Sp^p)}\leq c_{p,n} \left\|f^{(n)}\right\|_{\infty},
\end{align*}
from which we deduce \eqref{MOIBound1}. The inequality \eqref{MOIBound2} follows from the fact that $\|\cdot\|_{p_{i}}\le\|\cdot\|_{p}$, for $1\le i\le n$. This completes the proof of the theorem.
\end{proof}

Let $1<p<\infty$. Let $U$ and $V$ be two unitary operators on $\hilh$ and $U-V\in\Sp^p(\hilh)$. Then by \cite[Theorem 2]{AyCoSu16} for every $f\in C^1(\cir)$, $f(U)-f(V)\in\Sp^{p}(\hilh)$. Moreover we have the following formula
\begin{align}\label{PerturbationFormula1}
f(U)-f(V)=\left[\Gamma^{U,V}(f^{[1]})\right](U-V).
\end{align}

We will prove a higher order counterpart of this result, which will allow us to express differences of multiple operator integrals of the form
\begin{align*}
\left[\Gamma^{U_1,\ldots,U_{i-1},U,U_i,\ldots,U_{n-1}}(f^{[n-1]})\right](K_1,\ldots,K_{n-1})-\left[\Gamma^{U_1,\ldots,U_{i-1},V,U_i,\ldots,U_{n-1}}(f^{[n-1]})\right](K_1,\ldots,K_{n-1})
\end{align*}
as a multiple operator integral associated with the function $f^{[n]}$, when $f\in C^{n}(\cir)$ and $U-V\in\Sp^p(\hilh)$.

To that aim we need the following result, which establishes an approximation property of multiple operator integrals.

\begin{ppsn}\label{ApproxResult}
Let  $n\in\N$, $n\ge 2$, let $U_1,\ldots,U_n$ be unitary operators on $\mathcal{H}$ and for all $1\leq i\leq n$, let $(U_i^j)_{j\in \mathbb{N}}$ be a sequence of unitary operators on $\mathcal{H}$ converging to $U_i$ in the strong operator topology (SOT).
Then for any $\varphi\in C(\mathbb{T}^n)$ and for any $K_1,\ldots,K_{n-1}\in \mathcal{S}^2(\mathcal{H})$,
\begin{equation}\label{Approx1}
\underset{j\to\infty}{\lim}\left\|\left[\Gamma^{U_1^j,\ldots,U_n^j}(\varphi)\right](K_1,\ldots,K_{n-1})-\left[\Gamma^{U_1,\ldots,U_n}(\varphi)\right](K_1,\ldots,K_{n-1}) \right\|_2=0.
\end{equation}
\end{ppsn}

\begin{proof} 
Let $g\in C(\cir)$. Notice that for every $1\leq i \leq n$, $g(U_i^j)\overset{SOT}{\rightarrow} g(U_i)$. Indeed, by approximation and by linearity, it is enough to prove it when $g$ is a monomial, in which case it simply follows from the fact that $U_i^j\overset{SOT}{\to} U_i$.
	
It is worth noting that proving the proposition suffices when $K_1,\ldots,K_{n-1}$ are rank one operators, and since for any $\varphi\in C(\cir^{n})$, there is a sequence of trigonometric polynomials (in $n$ variables) uniformly converging to $\varphi$, it is enough to assume that $\varphi\in C(\cir)\otimes\cdots\otimes C(\cir)$. The rest of the proof follows from the same computations done in \cite[Proposition 3.1]{CoLeSkSu17}.
\end{proof}

In the next proposition we establish a standard perturbation formula for multiple operator integrals with respect to auxiliary unitary operators when $f\in C^{n}(\cir)$. It extends the function class of \cite[Lemma 2.4]{PoSkSu16}, where such representation was obtained for $f\in B^{n}_{\infty 1}(\cir)$ and generalizes \cite[Lemma 3.3(ii)]{PoSkSuTo17} from $\hilh=\ell^{2}_{d}$ to an arbitrary complex separable Hilbert space $\hilh$.

\begin{ppsn}\label{PerturbationFormula2}
Let $1<p<\infty$ and $n\ge 2$ be an integer. Let $U_1,\ldots,U_{n-1},U,V$ be unitary operators on $\hilh$ and assume that $U-V\in\Sp^p(\hilh)$. Let $f\in C^n(\cir)$. Then, for all $K_1,\ldots,K_{n-1}\in\Sp^p(\hilh)$ and for any $1\le i\le n$ we have
\begin{align*}
&\left[\Gamma^{U_1,\ldots,U_{i-1},U,U_i,\ldots,U_{n-1}}(f^{[n-1]})-\Gamma^{U_1,\ldots,U_{i-1},V,U_i,\ldots,U_{n-1}}(f^{[n-1]})\right](K_1,\ldots,K_{n-1})\\
&\hspace*{0.5cm}=\left[\Gamma^{U_1,\ldots,U_{i-1},U,V,U_i,\ldots,U_{n-1}}(f^{[n]})\right](K_1,\ldots,K_{i-1},U-V,K_i,\ldots,K_{n-1}).
\end{align*}
\end{ppsn}

\begin{proof} 
Let us start with the case $p=2$. As in the proof of Theorem \ref{MOISpBound}, we define $\varphi_k :=f\ast F_k$ where $F_k$ is the Fej\'er kernel. Since \begin{align*}
(\varphi_k)^{[n-1]}\underset{k\to\infty}{\to}f^{[n-1]}\quad\text{and}\quad(\varphi_k)^{[n]} \underset{k\to\infty}{\to}f^{[n]}
\end{align*} 
uniformly on $\mathbb{T}^n$ and $\mathbb{T}^{n+1}$ respectively, it is enough to prove the formula for $\varphi_k$ instead of $f$. Hence, we can assume that $f$ is a trigonometric polynomial and by linearity, we only have to treat the case when $f(z)=z^m$ for some $m\in\mathbb{Z}$. From now on, the proof simply consists in algebraic identities.
	
If $m\in\N$, it is straightforward to check that $f^{[n-1]}$ can be written as a finite sum $f^{[n-1]}=\sum a_k\psi_{\tilde{k}}$, where $\psi_{\tilde{k}}:=\psi_{k_1,\ldots,k_n}: (u_1,\ldots,u_n)\mapsto u_1^{k_1}u_2^{k_2}\cdots u_n^{k_n}$, $k_j\ge 0$, see for instance \cite[Lemma 2.1]{Sk17}. In such case, we have, for $1\le i\le n$,
\begin{align*}
f^{[n]}(u_1,\ldots,u_{n+1})
&=\sum a_k u_1^{k_1}\cdots u_{i-1}^{k_{i-1}} \frac{u_i^{k_i} - u_{i+1}^{k_i}}{u_i-u_{i+1}} u_{i+2}^{k_{i+1}} \cdots  u_{n+1}^{k_{n}}\\
&=\sum a_k\sum_{p=0}^{k_i-1}u_1^{k_1} \cdots u_{i-1}^{k_{i-1}}u_i^pu_{i+1}^{k_i-1-p} u_{i+2}^{k_{i+1}}\cdots u_{n+1}^{k_{n}}.
\end{align*}
Hence, by definition of multiple operator integrals, we have
\begin{align*}
&\left[\Gamma^{U_1,\ldots,U_{i-1},U,V,U_i,\ldots,U_{n-1}}(f^{[n]})\right](K_1,\ldots,K_{i-1},U-V,K_i,\ldots,K_{n-1})\\
&=\sum a_k U_1^{k_1}K_1\cdots K_{i-2}U_{i-1}^{k_{i-1}}K_{i-1}\left(\sum_{p=0}^{k_i-1} U^{p}(U-V)V^{k_i-1-p}\right)K_{i}U_i^{k_{i+1}}K_{i+1}\cdots K_{n-1}U_{n-1}^{k_{n}}\\
&=\sum a_k U_1^{k_1}K_1\cdots K_{i-2}U_{i-1}^{k_{i-1}}K_{i-1}\left(U^{k_i}-V^{k_i}\right)K_{i}U_i^{k_{i+1}}K_{i+1}\cdots K_{n-1}U_{n-1}^{k_{n}},
\end{align*}
which in turn is equal to 
\begin{align*}
\left[\Gamma^{U_1,\ldots,U_{i-1},U,U_i,\ldots,U_{n-1}}(f^{[n-1]})\right](K_1,\ldots,K_{n-1})-\left[\Gamma^{U_1,\ldots,U_{i-1},V,U_i,\ldots,U_{n-1}}(f^{[n-1]})\right] (K_1,\ldots,K_{n-1}).
\end{align*}
If $m\le 0$, then $f^{[n-1]}$ can be written as a linear combination of functions of the form $(u_1,\ldots,u_n)\mapsto u_1^{-k_1}u_2^{-k_2}\cdots u_n^{-k_n}$, where $k_j\ge 0$, so that $f^{[n]}$ can be written as the linear combination of the functions \begin{align*}
(u_1,\ldots,u_{n+1})\mapsto-u_1^{-k_1}\cdots u_{i-1}^{-k_{i-1}}\left(\sum_{p=0}^{k_i-1}u_i^{-1-p}u_{i+1}^{-k_i+p}\right)u_{i+2}^{-k_{i+1}}\cdots u_{n+1}^{-k_{n}}
\end{align*}
(see \cite[Lemma 2.1]{Sk17}). We conclude the result with similar computations as in the previous case. Hence, we proved the formula when $p=2$.
	
Assume now that $1<p<2$. Since $\Sp^p(\hilh)\subset \Sp^2(\hilh)$, the formula holds true as well.	
	
Finally, assume that $2<p<\infty$. Recall that $\Sp^2(\hilh)$ is a dense subspace of $\Sp^p(\hilh)$ and that $\|\cdot\|_p\le\|\cdot\|_2$. Since $U-V\in\Sp^p(\hilh)$, there exists an operator $A\in\Sp^p_{sa}(\hilh)$ such that $V=e^{iA}U$. Let $(A_j)_j\subset\Sp^2_{sa}(\hilh)$ be a sequence of self-adjoint elements converging to $A$ in $\|\cdot\|_p$ and denote $V_j=e^{iA_j}U$. For every $j$, $V_{j}-U\in\Sp^2(\hilh)$ so that, by the case $p=2$, for every $K_1,\ldots,K_{n-1}\in\Sp^2(\hilh)$,
\begin{align}
\nonumber&\left[\Gamma^{U_1,\ldots,U_{i-1},U,U_i,\ldots,U_{n-1}}(f^{[n-1]})-\Gamma^{U_1,\ldots,U_{i-1},V_j,U_i,\ldots,U_{n-1}}(f^{[n-1]})\right](K_1,\ldots,K_{n-1})\\
\label{ref12}&\hspace*{0.5cm}=\left[\Gamma^{U_1,\ldots,U_{i-1},U,V_j,U_i,\ldots,U_{n-1}}(f^{[n]})\right](K_1,\ldots,K_{i-1},U-V_j,K_i,\ldots,K_{n-1}).
\end{align}
Note that $V_j-V=(e^{iA_j}-e^{iA})U$ so that $V_j-V\in\Sp^p(\hilh)$ and by Duhamel's formula (see, e.g., \cite[Lemma 5.2]{AzCaDoSu09}) we have the following
$$\|V_j-V\|_p \leq \|A_j-A\|_p,$$ 
which goes to $0$ as $j\to\infty$. Let $\epsilon>0$. Let $N$ be large enough so that
\begin{equation}\label{ref3}
\forall j\ge N,\ \|V_j-V\|_p <\epsilon,
\end{equation}
and let $B\in\Sp^2(\hilh)$ be such that
\begin{equation}\label{ref4}
\|U-V-B\|_p <\epsilon.
\end{equation}
By Proposition \ref{ApproxResult}, the term
$$\left[ \Gamma^{U_1,\ldots,U_{i-1},V_j,U_i,\ldots,U_{n-1}}(f^{[n-1]})\right](K_1,\ldots,K_{n-1})$$
converges in $\Sp^2(\hilh)$ to
$$\left[\Gamma^{U_1,\ldots,U_{i-1},V,U_i,\ldots,U_{n-1}}(f^{[n-1]})\right](K_1,\ldots,K_{n-1}).$$
To deal with the convergence of the right-hand side of \eqref{ref12}, let us denote, for $j\in\N$,
$$\Gamma_j :X\in\Sp^p(\hilh)\mapsto\left[\Gamma^{U_1,\ldots, U_{i-1},U,V_j,U_i,\ldots,U_{n-1}}(f^{[n]})\right](K_1, \ldots, K_{i-1}, X, K_i,\ldots,K_{n-1})$$
and
$$\Gamma :X\in\Sp^p(\hilh)\mapsto\left[\Gamma^{U_1,\ldots,U_{i-1},U,V,U_i,\ldots,U_{n-1}}(f^{[n]})\right](K_1,\ldots,K_{i-1}, X, K_i,\ldots,K_{n-1}).$$
We have the following identity in $\Sp^p(\hilh)$: for every $j\ge N$,
\begin{align*}
&\Gamma_j(U-V_j)-\Gamma(U-V)\\
&=\Gamma_j(V-V_j)+\Gamma_j(U-V-B)+(\Gamma_j(B)-\Gamma(B))+\Gamma(B-(U-V)).
\end{align*}
By Proposition \ref{ApproxResult}, there exists $N'\in\N$ such that
$$\forall j\ge N',\ \|\Gamma_j(B)-\Gamma(B)\|_p\le\|\Gamma_j(B)-\Gamma(B)\|_2 <\epsilon.$$
Next, by Theorem \ref{MOISpBound}, $\left\lbrace \Gamma_j\mid j\in\N\right\rbrace\cup\left\lbrace\Gamma\right\rbrace\subset\mathcal{B}(\Sp^p(\hilh))$ is uniformly bounded. Hence, there exists a constant $C$ such that, using \eqref{ref3} and \eqref{ref4}, 
$$\|\Gamma_j(V-V_j)\|_p\leq C\|V-V_j\|_p <C\epsilon, \ \|\Gamma_j(U-V-B)\|_p< C\epsilon\quad\text{and}\quad\|\Gamma(B-(U-V))\|_p<C\epsilon.$$
Hence, for every $j\ge\max\{N,N'\}$,
$$\|\Gamma_j(U-V_j)-\Gamma(U-V)\|_p\le(3C+1)\epsilon.$$
This concludes the proof when $K_1,\ldots,K_{n-1}\in\Sp^2(\hilh)$ and $U-V\in\Sp^p(\hilh),\,p>2$. By approximation and boundedness of operator integrals on $\Sp^p(\hilh)$ given by Theorem \ref{MOISpBound}, we obtain the general case. This completes the proof in full generality.
\end{proof} 

This section concludes with the following lemma, which is a straightforward consequence of Proposition \ref{PerturbationFormula2}.

\begin{lma}\label{PerturbationFormula3}
Let $1<p<\infty$ and $n\ge2$ be an integer. Let $f\in C^{n}(\cir)$. Let $U,V$ be two unitary operators with $U-V\in\Sp^{p}(\hilh)$, and let $K_1,\ldots,K_{n-1}\in\Sp^{p}(\hilh)$. Then, for $1\le k\le n$
\begin{align*}
&\left[\Gamma^{(U)^{k},(V)^{n-k}}(f^{[n-1]})-\Gamma^{(V)^n}(f^{[n-1]})\right](K_1,\ldots,K_{n-1})\\
&\hspace*{0.5cm}=\sum_{i=1}^{k}\left[\Gamma^{(U)^{i},(V)^{n-i+1}}(f^{[n]})\right](K_1,\ldots,K_{i-1},U-V,K_i,\ldots,K_{n-1}).
\end{align*}
\end{lma}

\begin{proof}
We have 
\begin{align}
\nonumber&\left[\Gamma^{(U)^{k},(V)^{n-k}}(f^{[n-1]})-\Gamma^{(V)^n}(f^{[n-1]})\right](K_1,\ldots,K_{n-1})\\
\label{ref13}&\hspace*{2.5cm}=\sum_{i=1}^{k}\left[\Gamma^{(U)^{i},(V)^{n-i}}(f^{[n-1]})-\Gamma^{(U)^{i-1},(V)^{n-i+1}}(f^{[n-1]})\right](K_1,\ldots,K_{n-1}).
\end{align}
Now applying Proposition \ref{PerturbationFormula2} to \eqref{ref13} we complete the proof.
\end{proof}

\section{Differentiability in $\Sp^p(\hilh)$, $1<p<\infty$}\label{Sec3}

\subsection{Statements of the main results}\label{Subsec3}\hfill\\

In this section, we state our main results on differentiability of functions of unitary operators in $\Sp^{p}$-norms, for $1<p<\infty$.

We start by defining the G\^ateaux and Fr\'echet differentiability of operator functions. In this paper, we explore the concept of higher order Fr\'echet and G\^ateaux differentiability of $f\in\text{Lip}(\cir)$ at $U(t)\in\mathcal{U}(\hilh)$ as discussed below.

Let $U:\R\rightarrow\mathcal{U}(\hilh)$ be such that the function $\tilde{U}:t\in\R\mapsto U(t)-U(0)\in\Sp^p(\hilh)$. Let $f\in\text{Lip}(\cir)$, $1<p<\infty$. By \cite[Theorem 2]{AyCoSu16}, the function 
\begin{align}
\varphi_{U,f,p}:t\in\R\mapsto f(U(t))-f(U(0))\in\Sp^{p}(\hilh)
\end{align}
is well-defined.

\begin{dfn}[G\^ateaux derivative]\label{Gateauxderivativedef}
Let $n\in\N$ and $\tilde{U}\in D^{n}(\R,\Sp^{p}(\hilh))$. A function $f\in\text{Lip}(\cir)$ is said to be $n$ times G\^ateaux $\Sp^{p}$-differentiable at $U(t)$ if the function $\varphi_{U,f,p}$ is $n-1$ times $\Sp^{p}$-differentiable in a neighborhood of $t$ and $$\varphi^{(n-1)}_{U,f,p} : \mathbb{R} \to \Sp^{p}(\hilh)$$ is $\Sp^{p}$-differentiable at $t$. The expression $\varphi^{(n)}_{U,f,p}(t)$ denotes the $n$th G\^ateaux $\Sp^{p}$-derivative of $f$ at $U(t)$.
\end{dfn}

Let $U,V\in\mathcal{U}(\hilh)$. By \enquote{$V$ in the $\Sp^{p}$-neighborhood of $U$} we mean that
\begin{align*}
V\in\left\{X\in\mathcal{U}(\hilh):X-U\in\mathcal{W}\right\},
\end{align*}
where $\mathcal{W}$ is a $\Sp^{p}$-neighborhood of $0$.
A function $f\in\text{Lip}(\cir)$ is said to be Fr\'echet $\Sp^{p}$-differentiable at $U\in\mathcal{U}(\hilh)$ if there exists a bounded operator
$$D^1_{F,p}f(U)\in\mathcal{B}(\Sp^{p}(\hilh))$$
satisfying
$$\|f(V)-f(U)-D^1_{F,p}f(U)(V-U)\|_p=o(\|V-U\|_p),$$
for $V$ in a $\Sp^p$-neighborhood of $U$. Note that $V$ being in such neighborhood means that $V=e^{iA}U$, where $A\in\Sp^{p}_{sa}(\hilh)$ and $\|e^{iA}-I\|_{p}$ is small enough.

\begin{dfn}[Fr\'echet derivative]\label{Frechetderivativedef}
Let $n\in\N$. We say that $f\in\text{Lip}(\cir)$ is $n$ times Fr\'echet $\Sp^{p}$-differentiable at $U\in \mathcal{U}(\hilh)$ if it is $n-1$ times Fr\'echet $\Sp^{p}$-differentiable in a $\Sp^p$-neighborhood of $U$ and there is a $n$-linear bounded operator
$$D^n_{F,p}f(U)\in\mathcal{B}_n(\Sp^{p}(\hilh))$$
satisfying, for every $X_1,\ldots,X_{n-1}\in\Sp^{p}(\hilh)$,
\begin{align*}
&\|(D_{F,p}^{n-1}f(V)-D_{F,p}^{n-1}f(U))(X_1,\ldots,X_{n-1})-D_{F,p}^nf(U)(X_1,\ldots,X_{n-1},V-U)\|_p\\
&=o(\|V-U\|_p)\|X_1\|_p\cdots\|X_{n-1}\|_p,
\end{align*}
as $\|V-U\|_p\to 0$, for $V$ in a $\Sp^p$-neighborhood of $U$.
\end{dfn}

Next, we say that $f$ is $n$ times continuously Fr\'echet $\Sp^{p}$-differentiable at $U\in \mathcal{\hilh}$ if it is $n$ times Fr\'echet $\Sp^{p}$-differentiable in a $\Sp^p$-neighborhood of $U$ and for every $X_1,\ldots,X_n\in\Sp^{p}(\hilh)$ and $V$ in a $\Sp^p$-neighborhood of $U$, 
$$\|D^{n}_{F,p}f(V)(X_1,\ldots,X_n)-D^{n}_{F,p}f(U)(X_1,\ldots,X_n)\|=o(1)\|X_1\|_p\cdots\|X_n\|_p,$$
as $\|V-U\|_p\to 0$.

We are now ready to state our first main result. Prior to that, we denote by $\text{Sym}_{k}$ the group of permutations of the set $\{1,\ldots,k\}$.

\begin{thm}\label{Frechetdiff}
Let $1<p<\infty$ and $n\in\N$. Let $f\in C^n(\cir)$. Then $f$ is $n$ times continuously Fr\'echet $\Sp^{p}$-differentiable at every $U\in\mathcal{U}(\hilh)$ and for every $1\le k\le n$, $X_1,\ldots,X_k\in\Sp^{p}(\hilh)$,
\begin{equation}\label{formdiff1}
D^{k}_{F,p}f(U)(X_1,\ldots,X_k) = \sum_{\sigma \in \text{\normalfont Sym}_k} \left[\Gamma^{(U)^{k+1}}(f^{[k]})\right](X_{\sigma(1)},\ldots,X_{\sigma(k)}).
\end{equation}
\end{thm}

\begin{rmrk}
The converse of the above Theorem \ref{Frechetdiff} holds true, namely, if $f$ is $n$ times continuously Fr\'echet $\Sp^{p}$-differentiable at every $U\in\mathcal{U}(\hilh)$, then $f\in C^n(\cir)$. This follows from similar computations as in {\normalfont\cite[Proposition 3.9 (ii)]{{LeSk20}}}.
\end{rmrk}

We now establish the $n$th order G\^ateaux derivative of $f:\cir\to\C$. The subsequent result shows that for $f\in C^{n}(\cir)$, the operator function $\R\ni t\mapsto f(U(t))-f(U(0))\in\Sp^p(\hilh)$ is $n$ times differentiable (resp. continuously differentiable) in $\Sp^{p}(\hilh)$ if $U(t)-U(0)\in\Sp^p(\hilh)$ belongs to $D^n(\R,\Sp^p(\hilh))$ (resp. $C^n(\R,\Sp^p(\hilh))$). Since we are working on the path $t\mapsto U(t)$ (which may not always be linear), the expression of the $n$th order G\^ateaux derivative of $f$ will be bizarre to obtain from the definition of the Fr\'echet derivative directly unlike in the linear path. Here, we provide an easy-to-get expression of the G\^ateaux derivative using Lemma \ref{DifferentiationLemma}. Furthermore, we express the $n$th order operator Taylor remainder corresponding to the unitary operator $U(t)$ as multiple operator integrals when $f\in C^{n}(\cir)$.

\begin{thm}\label{MainTheorem1}
Let $1<p<\infty$ and $n\in\N$. Let $U:\R\rightarrow\mathcal{U}(\hilh)$ be such that the function $\tilde{U}:t\in\R\mapsto U(t)-U(0)\in\Sp^p(\hilh)$ belongs to $D^n(\R,\Sp^p(\hilh))$. Let $f\in C^n(\cir)$. We consider the function
$$\varphi:t\in\R\mapsto f(U(t))-f(U(0))\in\Sp^p(\hilh).$$
Then the following assertions hold.
\begin{enumerate}[{\normalfont(i)}]
\item The function $\varphi\in D^n(\R,\Sp^{p}(\hilh))$ and for every integer $1\le k\le n$ and $t\in\R$,
\begin{align}
\label{DerivativeFormula1}&\varphi^{(k)}(t)=\sum_{m=1}^{k} \sum_{\substack{l_1,\ldots,l_m\ge 1 \\ l_1+\cdots+l_m=k}}\dfrac{k!}{l_1!\cdots l_m!}\left[\Gamma^{(U(t))^{m+1}}(f^{[m]})\right]\left(\tilde{U}^{(l_1)}(t),\ldots,\tilde{U}^{(l_m)}(t)\right).
\end{align}
In addition, if $\tilde{U}\in C^{n}(\R,\Sp^{p}(\hilh))$, then $\varphi\in C^{n}(\R,\Sp^{p}(\hilh))$.\vspace*{0.1cm}
\item The operator Taylor remainder defined by 
\begin{align}
\label{ref20}R_{n,f,U}(t):=f(U(t))-f(U(0))-\sum_{k=1}^{n-1} \dfrac{1}{k!}\varphi^{(k)}(0)
\end{align} 
satisfies, for any $t\in\R$,
\begin{equation}\label{TaylorFormula1}
R_{n,f,U}(t)=\sum_{m=1}^n\sum_{\substack{l_1,\ldots,l_m \ge 1 \\ l_1+\cdots+l_m=n}}\left[\Gamma^{U(t),(U(0))^{m}}(f^{[m]})\right]\left(R_{l_1,U}(t),\dfrac{\tilde{U}^{(l_2)}(0)}{l_2!},\ldots,\dfrac{\tilde{U}^{(l_m)}(0)}{l_m!}\right),
\end{equation}
where $R_{1,U}(t):=\tilde{U}(t)$ and for any $l_1\ge 1$,
\begin{equation*}
R_{l_{1},U}(t):=\tilde{U}(t)-\sum_{k=1}^{l_{1}-1} \dfrac{1}{k!}\tilde{U}^{(k)}(0).
\end{equation*}
\end{enumerate}
\end{thm}

Finally, the differentiability result stated below is the $\Sp^{p}$-analogue of \cite[Theorem 3.1]{PoSkSu16}. Note that, Theorem 3.1 in \cite{PoSkSu16} establishes the existence of the $n$th order G\^ateaux derivative (in operator norm) of $f$ under the assumption of $f\in B^{n}_{\infty 1}(\cir)$. We prove here that $f$ is actually $n$ times continuously G\^ateaux differentiable (in Schatten $p$-norm) under the assumption that $f\in C^{n}(\cir)$. In addition, we also express the $n$th order Taylor remainder as a sum of multiple operator integrals and deduce an $\Sp^{p/n}$-estimate in the case when $f\in C^{n}(\cir)$ for $1<n<p<\infty$. As discussed in the introduction, this result is already established in \cite[Theorem 4.2]{PoSkSuTo17} under the same assumption, here we give a simple proof of it without using the approximation of $f$ or the integral representation for the remainder of the Taylor approximation, as found in the aforementioned paper.

\begin{crl}\label{MainCorollary1}
Let $1<p<\infty$ and $n\in\N$. Assume that $A\in\Sp^{p}_{sa}(\hilh)$ and $U\in\mathcal{U}(\hilh)$. Let $f\in C^n(\cir)$. Let $U(t)=e^{itA}U, t\in\R$ and consider the function
$$\varphi:t\in\R\mapsto f(e^{itA}U)-f(U)\in\Sp^{p}(\hilh).$$
\begin{enumerate}[{\normalfont(i)}]
\item The function $\varphi\in C^n(\R,\Sp^p(\hilh))$ and for every integer $1\le k\le n$ and $t\in\R$,
\begin{equation}\label{DerivativeFormula2}
\varphi^{(k)}(t)=i^{k}\sum_{m=1}^{k}\sum_{\substack{l_1,\ldots,l_m\ge 1 \\ l_1+\cdots+l_m=k}}\dfrac{k!}{l_1!\cdots l_m!}\left[\Gamma^{(U(t))^{m+1}}(f^{[m]})\right]\left(A^{l_1}U(t),\ldots,A^{l_m}U(t)\right).
\end{equation}
Moreover, there exists a constant $\tilde{c}_{p,n}>0$ such that
\begin{align}\label{DerivativeBound}
\left\|\varphi^{(n)}(t)\right\|_{p}\le \tilde{c}_{p,n}\sum_{m=1}^{n}\left\|f^{(m)}\right\|_{\infty}\|A\|_{p}^{n}.
\end{align}
\item The operator Taylor remainder defined by
\begin{align}
\label{ref21}R_{n,f}(A,U):=f(e^{iA}U)-f(U) - \sum_{k=1}^{n-1} \dfrac{1}{k!} \varphi^{(k)}(0)
\end{align}
satisfies
\begin{equation}\label{TaylorFormula2}
R_{n, f}(A,U)=\sum_{m=1}^n\sum_{\substack{l_1,\ldots, l_m\ge 1 \\ l_1+\cdots+l_m=n}}\left[\Gamma^{e^{iA}U,(U)^{m}}(f^{[m]})\right] \left(\sum_{k=l_1}^{\infty}\dfrac{(iA)^k}{k!}U,\dfrac{(iA)^{l_2}}{l_2!}U,\ldots,\dfrac{(iA)^{l_m}}{l_m!}U\right).
\end{equation}
Moreover, if $1<n<p<\infty$, then there exists a constant $\tilde{c}_{p,n}>0$ such that
\begin{align}\label{TaylorEstimate}
\|R_{n, f}(A,U)\|_{\frac{p}{n}}\le \tilde{c}_{p,n}\sum_{m=1}^{n}\left\|f^{(m)}\right\|_{\infty}\|A\|_{p}^{n}.
\end{align}
\end{enumerate}
\end{crl}

\subsection{Some auxiliary lemmas}\label{Subsec4}\hfill\\

In this section, we will provide some technical results that will be used to prove our main results in the next section.

\begin{lma}\label{LemmaFrediff}
Let $1<p<\infty$ and $n\in\N$. Let $U_1,\ldots,U_{n+1},V_1,\ldots,V_{n+1}\in\mathcal{U}(\hilh)$ and $X_1, \ldots, X_n \in \Sp^{np}(\hilh)$. Let $f \in C^n(\cir)$. Then, for every $\epsilon>0$, there exists a constant $\delta>0$ such that, if $\|V_i-U_i\|\le\delta$ for every $1\leq i\le n+1$, then
$$\left\| \left[ \Gamma^{V_1,\ldots, V_{n+1}}(f^{[n]})-\Gamma^{U_1,\ldots, U_{n+1}}(f^{[n]}) \right](X_1,\ldots,X_n) \right\|_p \leq \epsilon \|X_1\|_{np} \cdots\|X_n\|_{np}.$$
\end{lma}

\begin{proof}
Fix $\epsilon >0$. We first prove the lemma when $f^{[n]}=Q$ is a trigonometric polynomial. By linearity, it suffices to consider the case $Q(z_1,\ldots,z_{n+1})=z_1^{m_1}\cdots z_{n+1}^{m_{n+1}}$, where $m_i\in\Z, 1\le i\le n+1$. Since
\begin{align*}
&\left[\Gamma^{V_1,\ldots,V_{n+1}}(Q)-\Gamma^{U_1,\ldots,U_{n+1}}(Q)\right](X_1,\ldots,X_n)\\
&=V_1^{m_1}X_1 V_2^{m_2}\cdots V_n^{m_n}X_nV_{n+1}^{m_{n+1}}-U_1^{m_1} X_1 U_2^{m_2}\cdots U_n^{m_n}X_n U_{n+1}^{m_{n+1}}\\
&=\sum_{k=1}^{n+1}\left(V_1^{m_1}\cdots V_{k-1}^{m_{k-1}}X_{k-1}V_k^{m_k}X_k U_{k+1}^{m_{k+1}}\cdots U_{n+1}^{m_{n+1}}-V_1^{m_1}\cdots V_{k-1}^{m_{k-1}}X_{k-1}U_k^{m_k}\cdots U_{n+1}^{m_{n+1}}\right)\\
&=\sum_{k=1}^{n+1}V_1^{m_1}\cdots V_{k-1}^{m_{k-1}}X_{k-1}(V_k^{m_k}-U_k^{m_k})X_k U_{k+1}^{m_{k+1}}\cdots U_{n+1}^{m_{n+1}}
\end{align*}
and $U_i, V_i$ are unitaries, by Theorem \ref{MOISpBound} we get 
\begin{align*}
&\left\|\left[\Gamma^{V_1,\ldots,V_{n+1}}(f^{[n]})-\Gamma^{U_1,\ldots,U_{n+1}}(f^{[n]})\right](X_1,\ldots,X_n)\right\|_{p}\\
&\hspace*{0.5cm}\le(n+1)\max_{1\le k\le n+1}\|V_k^{m_k}-U_k^{m_k}\|\,\|X_1\|_{np} \cdots\|X_n\|_{np}.
\end{align*}
 By simple algebraic manipulations we obtain that for every $r\in\Z$,
$$\|V^{r}-U^{r}\|\le |r|\cdot\|V-U\|.$$
It follows that there exists a constant $K>0$ (depending on $n$ and $m_1,\ldots,m_{n+1}$, that is, only on $Q$) such that
\begin{align*}
\left\|\left[\Gamma^{V_1,\ldots,V_{n+1}}(f^{[n]})-\Gamma^{U_1,\ldots,U_{n+1}}(f^{[n]})\right](X_1,\ldots,X_n)\right\|\le K\max_{1\le k\le n+1}\|V_k-U_k\|\|X_1\|_{np} \cdots\|X_n\|_{np}.
\end{align*}
Choosing $\delta=\frac{\epsilon}{K}$ we conclude the proof for the particular case of a trigonometric polynomial.
		
In the general case, note that, as in the proof of Theorem \ref{MOISpBound}, one can find a trigonometric polynomial $P$ such that
$$\left\|f^{(n)}-P^{(n)}\right\|_{\infty}\le C\epsilon,$$
where $C>0$ is a constant that will be defined later. By Theorem \ref{MOISpBound}, there exists a constant $c_{p,n}$ such that
\begin{equation}\label{ref23}
\left\|\left[\Gamma^{V_1,\ldots,V_{n+1}}(f^{[n]}-P^{[n]})\right](X_1,\ldots,X_n)\right\|_p\le c_{p,n}C\epsilon \|X_1\|_{np}\cdots\|X_n\|_{np},
\end{equation}
and
\begin{equation}\label{ref24}
\left\|\left[\Gamma^{U_1,\ldots,U_{n+1}}(P^{[n]}-f^{[n]})\right](X_1,\ldots,X_n)\right\|_p\le c_{p,n} C \epsilon\|X_1\|_{np}\cdots\|X_n\|_{np}.
\end{equation}	
Next, the first part of the proof yields the existence of a constant $\delta>0$ such that, if $\|V_i-U_i\|\le\delta$ for every $1\le i\le n+1$, then
\begin{equation}\label{ref25}
\left\|\left[\Gamma^{V_1,\ldots,V_{n+1}}(P^{[n]})-\Gamma^{U_1,\ldots,U_{n+1}}(P^{[n]}) \right](X_1,\ldots,X_n) \right\|_p \leq \frac{\epsilon}{3} \|X_1\|_{np} \cdots\|X_n\|_{np}.
\end{equation}	
Finally, the equality
\begin{align*}
&\Gamma^{V_1,\ldots,V_{n+1}}(f^{[n]})-\Gamma^{U_1,\ldots,U_{n+1}}(f^{[n]})\\
&=\Gamma^{V_1,\ldots,V_{n+1}}(f^{[n]}-P^{[n]})+\left(\Gamma^{V_1,\ldots,V_{n+1}}(P^{[n]})-\Gamma^{U_1,\ldots,U_{n+1}}(P^{[n]})\right)+\Gamma^{U_1,\ldots,U_{n+1}}(P^{[n]}-f^{[n]}),
\end{align*}
together with \eqref{ref23}, \eqref{ref24} and \eqref{ref25} yield
\begin{align*}
\left\|\left[\Gamma^{V_1,\ldots,V_{n+1}}(f^{[n]})-\Gamma^{U_1,\ldots,U_{n+1}}(f^{[n]})\right](X_1,\ldots,X_n)\right\|_p 
&\le\{2c_{p,n}\,C\epsilon + \frac{\epsilon}{3}\} \|X_1\|_{np} \cdots \|X_n\|_{np},
\end{align*}
and setting
$C:=\frac{1}{3c_{p,n}}$ concludes the proof of the lemma.
\end{proof}
 
The following corollary is a simple consequence of the above Lemma \ref{LemmaFrediff}.

\begin{crl}\label{ContinuityCorollary}
Let $1<p<\infty$, $n\in\N$, $n\ge 2$ and $1\le j\le n$. Let $U:\mathbb{R}\rightarrow \mathcal{U}(\hilh)$ be such that the function $\tilde{U}:t\in\mathbb{R}\mapsto U(t)-U(0)\in\Sp^p(\hilh)$ belongs to $C(\R,\Sp^p(\mathcal{H}))$, and let $K_1,\ldots,K_{n-1}\in\Sp^{p}(\hilh)$. Let $f\in C^{n-1}(\cir)$. Define $\psi:s\in\R\mapsto\Sp^{p}(\hilh)$ by
\begin{align*}
&\psi(s)=\left[\Gamma^{(U(s))^{j},(U(t))^{n-j}}(f^{[n-1]})\right](K_1,\ldots,K_{n-1}).
\end{align*}
Then $\psi$ is continuous at $t\in\R$. In addition, if $f\in C^{n}(\cir)$ we have for any $s,t\in\R$,
\begin{align}
\label{ref5}&\psi(s)-\psi(t)=\sum_{i=1}^{j}\left[\Gamma^{(U(s))^{i},(U(t))^{n-i+1}}(f^{[n]})\right]\left(K_1,\ldots,K_{i-1},U(s)-U(t),K_{i},\ldots,K_{n-1}\right).
\end{align}
\end{crl}

To proceed further we need the following lemma, which will help us to prove the differentiability result in Theorem \ref{MainTheorem1}.

\begin{lma}\label{DifferentiationLemma}
Let $1<p<\infty$ and $U:\R\rightarrow\mathcal{U}(\hilh)$ be such that the function $\tilde{U}:t\in\R\mapsto U(t)-U(0)\in\Sp^p(\hilh)$ is differentiable. Let $m\in \N$, let $V_1,\ldots,V_m:\R\rightarrow \mathcal{S}^p(\hilh)$ be differentiable functions and let $f\in C^{m+1}(\cir)$. We consider the function
$$\psi: t\in\R\mapsto\Sp^p(\hilh)$$
defined by
$$\psi(t)=\left[\Gamma^{(U(t))^{m+1}}(f^{[m]})\right](V_1(t), \ldots, V_m(t)).$$
Then $\psi$ belongs to $D^1(\R,\Sp^p(\hilh))$ and its derivative is given, for any $t\in\R$ by
\begin{align*}
\psi'(t) = 
& \sum_{k=1}^m \left[\Gamma^{(U(t))^{m+1}}(f^{[m]})\right](V_1(t),\ldots, V_{k-1}(t), V_k'(t), V_{k+1}(t), \ldots, V_m(t)) \\
& + \sum_{k=1}^{m+1} \left[\Gamma^{(U(t))^{m+2}}(f^{[m+1]})\right](V_1(t),\ldots, V_{k-1}(t), \tilde{U}'(t), V_{k}(t), \ldots, V_m(t)).
\end{align*}
If we further assume that $\tilde{U}, V_1, \ldots, V_m$ are continuously differentiable, then so is $\psi$.
\end{lma}

\begin{proof}
For any $s,t \in \mathbb{R}$, write
$$\psi(t)-\psi(s)=\psi_1(s,t)+\psi_2(s,t),$$
where
$$\psi_1(s,t)=\left[\Gamma^{(U(t))^{m+1}}(f^{[m]})\right](V_1(t),\ldots,V_m(t))-\left[\Gamma^{(U(t))^{m+1}}(f^{[m]})\right](V_1(s),\ldots,V_m(s))$$
and
$$\psi_2(s,t)=\left[\Gamma^{(U(t))^{m+1}}(f^{[m]})\right](V_1(s),\ldots,V_m(s))-\left[\Gamma^{(U(s))^{m+1}}(f^{[m]})\right](V_1(s),\ldots,V_m(s)).$$
We have
\begin{align*}
\psi_1(s,t)&=\sum_{k=1}^m \left(\left[\Gamma^{(U(t))^{m+1}}(f^{[m]})\right](V_1(s),\ldots,V_{k-1}(s),V_k(t),\ldots,V_m(t))\right.\\
& \ \ \ \ \ \ \ \ \ -\left.\left[\Gamma^{(U(t))^{m+1}}(f^{[m]})\right](V_1(s),\ldots,V_k(s),V_{k+1}(t),\ldots,V_m(t))\right)\\
&=\sum_{k=1}^m\left[\Gamma^{(U(t))^{m+1}}(f^{[m]})\right](V_1(s),\ldots,V_{k-1}(s),V_k(t)-V_k(s),V_{k+1}(t),\ldots,V_m(t)).
\end{align*}
By differentiability (and hence continuity) of $V_k,~k=1,\ldots,m$, we have
$$\underset{s\to t}{\lim} \ V_k(s)=V_k(t) \ \ \text{and} \ \ \underset{s\to t}{\lim} \ \dfrac{V_k(t)-V_k(s)}{t-s}=V_k'(t)$$
in $\Sp^p(\hilh)$. Since $\Gamma^{(U(t))^{m+1}}(f^{[m]})\in\mathcal{B}_{m}(\Sp^p(\hilh))$, this implies that
\begin{align*}
\dfrac{\psi_1(s,t)}{t-s}&=\sum_{k=1}^m\left[\Gamma^{(U(t))^{m+1}}(f^{[m]})\right]\left(V_1(s),\ldots, V_{k-1}(s),\dfrac{V_k(t)-V_k(s)}{t-s},V_{k+1}(t),\ldots,V_m(t)\right)\\
&\underset{s\to t}{\longrightarrow}\sum_{k=1}^m\left[\Gamma^{(U(t))^{m+1}}(f^{[m]})\right](V_1(t),\ldots,V_{k-1}(t),V_k'(t),V_{k+1}(t),\ldots,V_m(t))
\end{align*}
in $\Sp^{p}(\hilh)$.
For $\psi_2$, by Lemma \ref{PerturbationFormula3} we have
\begin{align}
\nonumber&\dfrac{\psi_2(s,t)}{t-s}\\
\label{ref9}&=\sum_{k=1}^{m+1}\left[\Gamma^{(U(t))^{k},(U(s))^{m+2-k}}(f^{[m+1]})\right]\left(V_1(s),\ldots,V_{k-1}(s),\frac{\tilde{U}(t)-\tilde{U}(s)}{t-s},V_{k}(t),\ldots,V_m(t)\right).
\end{align}	
For $1\le k\le m+1$ and any $s\neq t$, let
\begin{align}
\label{ref15}\Gamma_{k}(s)=\left[\Gamma^{(U(t))^{k},(U(s))^{m+2-k}}(f^{[m+1]})\right]\left(V_1(s),\ldots,V_{k-1}(s),\frac{\tilde{U}(t)-\tilde{U}(s)}{t-s},V_{k}(t),\ldots,V_m(t)\right).
\end{align}
Since $(\tilde{U}(t)-\tilde{U}(s))/(t-s)$ goes to $\tilde{U}'(t)$ in $\Sp^{p}(\hilh)$ as $s$ goes to $t$, by the uniform boundedness of $\left[\Gamma^{(U(t))^{k},(U(s))^{m+2-k}}(f^{[m+1]})\right]\in\mathcal{B}_{m+1}(\Sp^{p}(\hilh))$, we deduce that 
$$\lim_{s\to t}\Gamma_{k}(s)$$
exists if and only if
\begin{align*}
\lim_{s\to t}\left[\Gamma^{(U(t))^{k},(U(s))^{m+2-k}}(f^{[m+1]})\right](V_1(t),\ldots,V_{k-1}(t),\tilde{U}'(t),V_{k}(t),\ldots,V_m(t)),
\end{align*}
exists, and in that case, both limits are equal. By Corollary \ref{ContinuityCorollary}, the latter limit exists and it further reduces to
\begin{align*}
\lim_{s\to t}\Gamma_{k}(s)=\left[\Gamma^{(U(t))^{m+2}}(f^{[m+1]})\right](V_1(t),\ldots,V_{k-1}(t),\tilde{U}'(t),V_{k}(t),\ldots,V_m(t))
\end{align*}
in $\Sp^{p}(\hilh)$. This proves the first claim.
	
In the case when $\tilde{U},V_1,\ldots,V_m$ are continuously differentiable, a similar line of argument as in the proof of differentiability of $\psi$ and an application of Corollary \ref{ContinuityCorollary} to $\psi'$ shows that $\psi'$ is continuous. The details are left to the reader. This completes the proof.
\end{proof}

\subsection{Proofs of the main results}\label{Subsec5}\hfill\\

We now turn to the proofs of the main results of this article, which were outlined in Subsection \ref{Subsec3}.

\begin{proof}[Proof of Theorem $\ref{Frechetdiff}$] In the proof, we fix $U\in\mathcal{U}(\hilh)$, $V$ will be an element of a $\Sp^p$-neighborhood of $U$ and $X_1,\ldots,X_{n-1}$ will denote elements of $\Sp^{p}(\hilh)$.
	
We prove the result by induction on $n$. The argument for the base case $n=1$ is the same as for a general integer $n$, hence we assume that this case has been proven.
	
Assume now that the result holds true for functions in $C^{n-1}(\cir)$, $n\geq 2$ and let $f\in C^n(\cir)$. In particular, $f\in C^{n-1}(\cir)$ so $f$ is $n-1$ times continuously Fr\'echet $\Sp^{p}$-differentiable at $U$ and Formula \eqref{formdiff1} holds true for every $k=1,\ldots,n-1$. We now show that $f$ is $n$ times continuously Fr\'echet $\Sp^{p}$-differentiable at $U$ and that \eqref{formdiff1} holds true for $k=n$. First, let us denote
$$T_1=(D_{F,p}^{n-1}f(V)-D_{F,p}^{n-1}f(U))(X_1,\ldots,X_{n-1}).$$
Note that, by assumption and by Lemma \ref{PerturbationFormula3},
\begin{align*}
T_1&=\sum_{\sigma\in\text{Sym}_{n-1}}\left[\Gamma^{(V)^{n}}(f^{[n-1]})-\Gamma^{(U)^{n}}(f^{[n-1]})\right](X_{\sigma(1)},\ldots,X_{\sigma(n-1)})\\
&=\sum_{\sigma\in\text{Sym}_{n-1}}\sum_{i=1}^{n}\left[\Gamma^{(V)^{i},(U)^{n-i+1}}(f^{[n]})\right](X_{\sigma(1)},\ldots,X_{\sigma(i-1)},V-U,X_{\sigma(i)},\ldots,X_{\sigma(n-1)}).
\end{align*}
Next, we let $X_n:=V-U\in\Sp^{p}(\hilh)$. Notice that
\begin{align*}
T_2&:=\sum_{\sigma\in\text{Sym}_n}\left[\Gamma^{(U)^{n+1}}(f^{[n]})\right](X_{\sigma(1)},\ldots,X_{\sigma(n)})\\
&=\sum_{i=1}^n\sum_{\sigma\in\text{Sym}_{n-1}}\left[\Gamma^{(U)^{n+1}}(f^{[n]})\right](X_{\sigma(1)},\ldots,X_{\sigma(i-1)},V-U,X_{\sigma(i)},\ldots,X_{\sigma(n-1)}),
\end{align*}
so that
\begin{align*}
&\|T_1-T_2\|_p\\ 
&\le\sum_{i=1}^n\sum_{\sigma\in\text{Sym}_{n-1}}\left\|\left[(\Gamma^{(V)^{i},(U)^{n-i+1}}-\Gamma^{(U)^{n+1}})(f^{[n]})\right](X_{\sigma(1)},\ldots,X_{\sigma(i-1)},V-U,X_{\sigma(i)},\ldots,X_{\sigma(n-1)})\right\|_p.
\end{align*}
Let $\epsilon>0$, $1\le i\le n$ and $\sigma\in\text{Sym}_{n-1}$. By Lemma \ref{LemmaFrediff}, there exists $\delta>0$ such that, if $$(\|V-U\| \le)\ \|V-U\|_{p}\le\delta,$$
then
\begin{align*}
&\left\|\left[(\Gamma^{(V)^{i},(U)^{n-i+1}}-\Gamma^{(U)^{n+1}})(f^{[n]})\right](X_{\sigma(1)},\ldots,X_{\sigma(i-1)},V-U,X_{\sigma(i)},\ldots,X_{\sigma(n-1)})\right\|_p\\
&\hspace*{0.5cm}\le\epsilon\|V-U\|_{np} \|X_1\|_{np}\cdots\|X_{n-1}\|_{np} \\
&\hspace*{0.5cm}\le\epsilon\|V-U\|_{p} \|X_1\|_{p} \cdots \|X_{n-1}\|_{p}.
\end{align*}
By linearity, this concludes the proof of the Fr\'echet $\Sp^{p}$-differentiability. The fact that $f$ is continuously Fr\'echet $\Sp^{p}$-differentiable now follows easily from \eqref{formdiff1}, linearity, and Lemma \ref{LemmaFrediff}.
\end{proof}

\begin{proof}[Proof of Theorem $\ref{MainTheorem1}$]

The assumption on $f$ ensures that $f$ is $n$ times continuously Fr\'echet $\Sp^{p}$-differentiable at every $U(t)\in\mathcal{U}(\hilh)$.\vspace*{0.2cm}

(i) We prove the first claim by induction on $k$, $1\le k\le n$. Let $s,t\in\R$. The $n$th order Fr\'echet $\Sp^{p}$-differentiability of $f$ at $U(t)$ implies that it is $n-1$ times Fr\'echet $\Sp^{p}$-differentiable in a $\Sp^p$-neighborhood of $U(t)$ and there is a $n$-linear bounded operator
$$D^n_{F,p}f(U(t))\in\mathcal{B}_n(\Sp^{p}(\hilh))$$
such that for every $X_1,\ldots,X_{n-1}\in\Sp^{p}(\hilh)$,
\begin{align*}
&\|(D_{F,p}^{n-1}f(U(s))-D_{F,p}^{n-1}f(U(t)))(X_1,\ldots,X_{n-1})-D_{F,p}^nf(U(t))(X_1,\ldots,X_{n-1},U(s)-U(t))\|_p\\
&=o(\|U(s)-U(t)\|_p)\|X_1\|_p\cdots\|X_{n-1}\|_p,
\end{align*}
as $\|U(s)-U(t)\|_p\to 0$. For $n=1$ there exists a  bounded linear operator $$D^{1}_{F,p}f(U(t))\in\mathcal{B}(\Sp^{p}(\hilh))$$
such that
\begin{align*}
\|f(U(s))-f(U(t))-D^{1}_{F,p}f(U(t))(U(s)-U(t))\|_{p}=o(\|U(s)-U(t)\|_{p})
\end{align*}
as $\|U(s)-U(t)\|_{p}\to0$.\\
Again note that 
\begin{align}
\nonumber&\left\|f(U(s))-f(U(t))-(s-t)D^{1}_{F,p}f(U(t))(\tilde{U}'(t))\right\|_{p}\\
\nonumber&\le\left\|f(U(s))-f(U(t))-D^{1}_{F,p}f(U(t))(U(s)-U(t))\right\|_{p}\\
\label{ref26}&\quad+\left\|(s-t)\left[D^{1}_{F,p}f(U(t))\left(\frac{U(s)-U(t)}{s-t}\right)-D^{1}_{F,p}f(U(t))(\tilde{U}'(t))\right]\right\|_{p}.
\end{align}
Now, the fact that $\tilde{U}\in D^{n}(\R,\Sp^{p}(\hilh))$ together with the inequality \eqref{ref26} imply that $f$ is G\^ateaux $\Sp^{p}$-differentiable at $U(t)$ along with its G\^ateaux derivative is $$\varphi'(t)=\frac{d}{ds}\bigg|_{s=t}f(U(s))=D_{F,p}^{1}f(U(t))(\tilde{U}'(t)),$$ 
where
\begin{align*}
D_{F,p}^{1}f(U(t))(\tilde{U}'(t))=\left[\Gamma^{(U(t))^{2}}(f^{[1]})\right](\tilde{U}'(t)).
\end{align*}
This proves the base of the induction.\\
Now assume that for $1\le k\le n-1$, $f$ is $k$ times G\^ateaux $\Sp^{p}$-differentiable, that is $\varphi\in D^{k}(\R,\Sp^{p}(\hilh))$, and
\begin{align}
\nonumber\varphi^{(k)}(t)&=\sum_{m=1}^{k} \sum_{\substack{l_1,\ldots,l_m\ge 1 \\ l_1+\cdots+l_m=k}}\dfrac{k!}{l_1!\cdots l_m!}\left[\Gamma^{(U(t))^{m+1}}(f^{[m]})\right]\left(\tilde{U}^{(l_1)}(t),\ldots,\tilde{U}^{(l_m)}(t)\right)\\
\nonumber&=\sum_{m=1}^{k-1} \sum_{\substack{l_1,\ldots,l_m\ge 1 \\ l_1+\cdots+l_m=k}}\dfrac{k!}{l_1!\cdots l_m!}\left[\Gamma^{(U(t))^{m+1}}(f^{[m]})\right]\left(\tilde{U}^{(l_1)}(t),\ldots,\tilde{U}^{(l_m)}(t)\right)\\
\nonumber&\quad+D_{F,p}^{k}f(U(t))\left((\tilde{U}'(t))^{k}\right)\\
\label{ref28}&:=[A]+[B].
\end{align}
To check the $(k+1)$th order G\^ateaux differentiability of $f$ it is enough to check the G\^ateaux differentiability of $[A]$ and $[B]$. Note that, the $n$th order Fr\'echet $\Sp^{p}$-differentiability of $f$ and $\tilde{U}\in D^{n}(\R,\Sp^{p}(\hilh))$ ensures that $[B]$ is G\^ateaux $\Sp^{p}$-differentiable, and
\begin{align}
\nonumber\frac{d}{ds}\bigg\vert_{s=t}D_{F,p}^{k}f(U(s))\left((\tilde{U}'(s))^{k}\right)&=\sum_{i=1}^{k}D_{F,p}^{k}f(U(t))\bigg(\tilde{U}'(t),\ldots,\tilde{U}'(t),\underbrace{\tilde{U}''(t)}_{i},\tilde{U}'(t),\ldots,\tilde{U}'(t)\bigg)\\
\label{ref29}&\quad+D_{F,p}^{k+1}f(U(t))\left((\tilde{U}'(t))^{k+1}\right).
\end{align}
On the other hand, the G\^ateaux $\Sp^{p}$-differentiability of $[A]$ is justified by Lemma \ref{DifferentiationLemma}. Moreover, applying Lemma \ref{DifferentiationLemma} in $[A]$, together with \eqref{ref29} give $\varphi^{(k)}\in D^1(\R,\Sp^p(\hilh))$, that is, $f$ is $(k+1)$ times G\^ateaux $\Sp^{p}$-differentiable, and
\begin{align*}
&\varphi^{(k+1)}(t)=\\
&k!\sum_{m=1}^k\sum_{\substack{l_1,\ldots,l_m\ge 1 \\ l_1+\cdots+l_m=k}}\dfrac{1}{l_1!\cdots l_m!}\sum_{p=1}^m\left[\Gamma^{(U(t))^{m+1}}(f^{[m]})\right]\left(\tilde{U}_t^{(l_1)},\ldots,\tilde{U}_t^{(l_{p-1})},\tilde{U}_t^{(l_p+1)},\tilde{U}_t^{(l_{p+1})},\ldots,\tilde{U}_t^{(l_m)}\right)\\
&+k!\sum_{m=1}^k\sum_{\substack{l_1,\ldots,l_m\ge 1 \\ l_1+\cdots+l_m=k}}\dfrac{1}{l_1!\cdots l_m!} \sum_{p=1}^{m+1}\left[\Gamma^{(U(t))^{m+2}}(f^{[m+1]})\right]\left(\tilde{U}_t^{(l_1)},\ldots,\tilde{U}_t^{(l_{p-1})},\tilde{U}_t',\tilde{U}_t^{(l_p)},\ldots,\tilde{U}_t^{(l_m)}\right) \\
&=k!\,(\varphi_1(t)+\varphi_2(t)),
\end{align*}
where $\tilde{U}_t^{(l)}$ stands for $\tilde{U}^{(l)}(t)$. Now by employing a similar combinatorial reasoning as demonstrated in the proof of \cite[Theorem 5.3.4]{SkToBook}, we conclude the induction on $k$.
 
This proves the Formula \eqref{DerivativeFormula1} and $\varphi\in D^{n}(\R,\Sp^{p}(\hilh))$.
Furthermore, if we assume that $\tilde{U}\in C^{n}(\R,\Sp^{p}(\hilh))$, then an application of Corollary \ref{ContinuityCorollary} shows that $\varphi^{(n)}\in C(\R,\Sp^{p}(\hilh))$.\vspace*{0.2cm} 
 
(ii) We obtain the representation \eqref{TaylorFormula1}  for the Taylor reminder $R_{k,f,U}(t)$ by induction on $k$. The base case $k=1$ follows from \eqref{PerturbationFormula1}. Notice that the assumption on $f$ ensures, by Theorem \ref{MainTheorem1}(i), that $\varphi\in D^{n}(\R,\Sp^{p}(\hilh))$. Further the inductive step is proved along the lines of the one in \cite[Theorem 5.4.1(ii)]{SkToBook}. Details are left to the reader. This completes the proof of the theorem.
\end{proof}

Finally, we now turn to the proof of Corollary \ref{MainCorollary1}.
\begin{proof}[Proof of Corollary $\ref{MainCorollary1}$]
Let $U(t)=e^{itA}U$, $t\in\R$, where $A\in\Sp^{p}_{sa}(\hilh)$ and define $$\tilde{U}(t):=e^{itA}U-U,\,t\in\R.$$
{\normalfont(i)} A routine calculation provides that for every $n\in\N$ we have $\tilde{U}\in C^{n}(\R,\Sp^{p}(\hilh))$, and
\begin{align}
\label{ref18}\tilde{U}^{(n)}(t)=(iA)^{n}U(t),\ \ t\in\R.
\end{align}
Now using Theorem \ref{MainTheorem1}(i) we prove \eqref{DerivativeFormula2} along with $\varphi\in C^{n}(\R,\Sp^{p}(\hilh))$. 

By Theorem \ref{MOISpBound}, there exists $c_{p,m}>0$ such that
\begin{align}
\label{ref22}\left\|\left[\Gamma^{(U(t))^{m+1}}(f^{[m]})\right]\left(A^{l_{1}}U(t),\ldots,A^{l_{m}}U(t)\right)\right\|_{p}\le c_{p,m}\left\|f^{(m)}\right\|_{\infty}\|A\|_{p}^{n}
\end{align}
for every $1\le m\le n$. Hence in view of \eqref{DerivativeFormula2} and \eqref{ref22} we have the existence of a constant $\tilde{c}_{p,n}>0$ such that for $1<p<\infty$,
\begin{align}
\left\|\varphi^{(n)}(t)\right\|_{p}\le \tilde{c}_{p,n}\sum_{m=1}^{n}\left\|f^{(m)}\right\|_{\infty}\|A\|_{p}^{n}.
\end{align}
{\normalcolor(ii)} Note that $R_{n,f}(A,U)=R_{n,f.U}(1)$ (see for instance \eqref{ref20} and \eqref{ref21}). Therefore \eqref{TaylorFormula1} along with \eqref{ref18} yield \eqref{TaylorFormula2}. 

Now assume that $1<n<p<\infty$ and $f\in C^{n}(\cir)$. By \cite[(4.2)]{PoSkSu16}, we have an integral representation for the remainder of the Taylor approximation
\begin{align*}
\tilde{U}(1)-\sum_{k=1}^{l_{1}-1} \dfrac{1}{k!}\tilde{U}^{(k)}(0)=\dfrac{1}{(l_{1}-1)!}\int_{0}^{1}(1-t)^{l_{1}-1}\tilde{U}^{(l_{1})}(t)dt,
\end{align*}
which along with \eqref{ref18} further implies the following estimate
\begin{align}
\label{ref19}\left\|\sum_{k=l_1}^{\infty}\dfrac{(iA)^k}{k!}U\right\|_{\frac{p}{l_{1}}}=\left\|\tilde{U}(1)-\sum_{k=1}^{l_{1}-1} \dfrac{1}{k!}\tilde{U}^{(k)}(0)\right\|_{\frac{p}{l_{1}}}\le\dfrac{1}{l_{1}!}\|A\|_{p}^{l_{1}}.
\end{align}
Therefore using \eqref{ref19} and Theorem \ref{MOISpBound} in \eqref{TaylorFormula2}, we have the existence of a constant $\tilde{c}_{p,n}>0$ such that
\begin{align*}
\|R_{n,f}(A,U)\|_{\frac{p}{n}}\le\tilde{c}_{p,n}\sum_{m=1}^{n}\left\|f^{(m)}\right\|_{\infty}\|A\|_{p}^{n}.
\end{align*}
This completes the proof of the corollary.
\end{proof}

\noindent\textit{Acknowledgment}: A. Chattopadhyay is supported by the Core Research Grant (CRG), File No: CRG/2023/004826, of SERB. S. Giri acknowledges the support by the Prime Minister's Research Fellowship (PMRF), Government of India. C. Pradhan acknowledges support from the IoE post-doctoral fellowship at IISc Bangalore. The authors would like to heartily thank the anonymous referee for a careful
reading of the manuscript and helpful suggestions.

\end{document}